\documentclass[11pt]{article}
\usepackage[margin=1in]{geometry}
\usepackage[utf8]{inputenc}
\usepackage[T1]{fontenc}
\usepackage[english]{babel}
\usepackage{lmodern}
\usepackage{graphicx}
\usepackage{wrapfig}
\usepackage{amsmath}
\usepackage{caption}
\usepackage{paralist}
\usepackage{enumitem}
\usepackage{threeparttable}
\usepackage{subcaption}

\usepackage[linesnumbered,ruled,vlined]{algorithm2e}

\graphicspath{{./figures/}}

\usepackage[usenames]{color}
\definecolor{hellblau}{rgb}{0.2,0.4,1}
\definecolor{dunkelblau}{rgb}{0,0,0.8}
\definecolor{dunkelgruen}{rgb}{0,0.5,0}
\usepackage[
pdftex,
colorlinks,
linkcolor=dunkelblau,
urlcolor=dunkelblau,
citecolor=dunkelgruen,
bookmarks=true,
linktocpage=true,
pdfsubject={}
]{hyperref}
\usepackage{pdfpages}
\usepackage{amsmath}
\usepackage{amsthm}
\allowdisplaybreaks
\usepackage{amsfonts}
\theoremstyle{plain}
\newtheorem{satz}{Satz}[]
\newtheorem{theorem}[satz]{Theorem}

\newtheorem{proposition}[satz]{Proposition}
\newtheorem{corollary}[satz]{Corollary}	
\theoremstyle{remark}

\theoremstyle{definition}

\newtheorem*{conjecture}{Conjecture}

\begin{document}
\title
{Improved asymptotic upper bounds for the minimum number of pairwise distinct longest cycles in regular graphs}
\author{
{\sc Jorik JOOKEN\footnote{Department of Computer Science, KU Leuven Campus Kulak-Kortrijk, 8500 Kortrijk, Belgium}}
\footnote{E-mail address: jorik.jooken@kuleuven.be}}
\date{}

\maketitle
\begin{center}
\begin{minipage}{125mm}
{\bf Abstract.}
We study how few pairwise distinct longest cycles a regular graph can have under additional constraints. For each integer $r \geq 5$, we give exponential improvements for the best asymptotic upper bounds for this invariant under the additional constraint that the graphs are $r$-regular hamiltonian graphs. Earlier work showed that a conjecture by Haythorpe on a lower bound for this invariant is false because of an incorrect constant factor, whereas our results imply that the conjecture is even asymptotically incorrect. Motivated by a question of Zamfirescu and work of Chia and Thomassen, we also study this invariant for non-hamiltonian 2-connected $r$-regular graphs and show that in this case the invariant can be bounded from above by a constant for all large enough graphs, even for graphs with arbitrarily large girth.
\smallskip


\smallskip


\end{minipage}
\end{center}

\vspace{1cm}

\section{Introduction}

The circumference $c(G)$ of a graph $G$ (the length of its longest cycle) is a well-studied parameter. A great number of interesting papers revolve around how small the number of pairwise distinct cycles with length $c(G)$, i.e. longest cycles, can be for various graph classes. A natural question is how this invariant is affected by the degrees of vertices in $G$. For many graph classes, the proven or conjectured lower bounds for this invariant are constants, whereas the upper bounds are often one of two extremes: they are either constants or exponential in the number of vertices (as will also be the case in the current paper).

We briefly discuss a number of important results from the literature on this topic that show how the current paper is connected with previous research. Tutte~\cite{Tu46} presented a theorem by Smith that every edge occurs in an even number of hamiltonian cycles of a hamiltonian cubic graph. From this, one can derive that every hamiltonian cubic graph has at least three pairwise distinct hamiltonian cycles. Thomason~\cite{Th78} later used his lollipop technique to generalize this lower bound to all hamiltonian graphs in which all vertices have an odd degree. For cubic graphs, there exist infinite families of graphs containing precisely three pairwise distinct hamiltonian cycles~\cite{Sc89,GMZ20}, but for other $r$-regular hamiltonian graphs with $r \geq 5$, the best known families contain an exponential number of hamiltonian cycles~\cite{Ha18, Za22}. Sheehan's famous conjecture~\cite{Sh75} from 1975 states that every hamiltonian 4-regular graph contains at least two distinct hamiltonian cycles. This longstanding conjecture remains unsolved until today, but luckily there is progress. For example, Thomassen~\cite{Th98} showed that every hamiltonian $r$-regular graph, with $r \geq 300$, contains at least two distinct hamiltonian cycles, whereas Haxell, Seamone and Verstraete~\cite{HSV07} improved this result to $r \geq 23$. Zamfirescu~\cite{Za22} recently described an infinite family of 4-regular graphs containing precisely 144 hamiltonian cycles, which is the best known upper bound. Gir\~{a}o, Kittipassorn and Narayanan~\cite{GKN19} proved that every graph with minimum vertex degree 3 having a hamiltonian cycle contains another long cycle, namely one with length at least $n-cn^{\frac{4}{5}}$, where $n$ is the order of the graph and $c$ is a constant. 

Apart from these examples, we also know of infinite families in other classes of graphs containing precisely one longest cycle, which is clearly the best possible. Goedgebeur, Meersman and Zamfirescu~\cite{GMZ20} reported on an infinite family of graphs with genus one that have precisely one hamiltonian cycle and at most one vertex of degree two. Entringer and Swart~\cite{ES80} described infinite families of graphs containing a unique hamiltonian cycle in which all vertices have degree 3 or 4, whereas Fleischner~\cite{Fl14} described such graphs in which the degrees are all 4 or 14. Recently, this result was strengthened by Brinkmann~\cite{Br22} to several sets of vertex degrees $D$, namely if the minimum of $D$ is a) 2 or b) 3 and $D$ contains an even degree or c) 4 and all other elements in $D$ are at least 10. Chia and Thomassen~\cite{CT12} showed the existence of an infinite family of non-hamiltonian cubic graphs with connectivity 2 having a unique longest cycle, whereas Zamfirescu~\cite{Za22} described such a connectivity 3 family. Zamfirescu also showed the existence of an infinite family of non-hamiltonian $r$-regular graphs with connectivity 1 having a unique longest cycle, for each integer $r \geq 3$.

This motivates us to introduce the following definitions. We define $h(n,r)$ as the minimum number of pairwise distinct longest cycles over all hamiltonian $r$-regular graphs on $n$ vertices if such graphs exist (note that such graphs exist when $n \geq r+1 \geq 3$ and $nr$ is even). Otherwise, we define $h(n,r)$ to be equal to $h(n+1,r)$ (this choice will make statements about asymptotic behaviour less cumbersome). Note that every longest cycle in a hamiltonian graph is of course a hamiltonian cycle and every hamiltonian graph is 2-connected. In a similar fashion, we define $h_2(n,r,g)$ as the minimum number of pairwise distinct longest cycles over all non-hamiltonian 2-connected $r$-regular graphs on $n$ vertices with girth equal to $g$ if such graphs exist (and equal to $h_2(n+1,r,g)$ otherwise).

In the current paper, we will be interested in asymptotic upper bounds for these two functions as the orders of the graphs grow (and all other parameters are fixed). All graphs in this paper are simple undirected graphs (loops and multiple edges are not allowed). The rest of this paper is structured as follows. In Section~\ref{sec:ham}, we give exponential improvements for the current best asymptotic upper bounds for $h(n,r)$ for all integers $r \geq 5$ by a constructive proof. This result also implies that a conjecture by Haythorpe~\cite{Ha18} on lower bounds for $h(n,r)$ is false in a strong sense, i.e. asymptotically. The new upper bounds are exponential in the number of vertices, but have a smaller base for the exponent than the previous best upper bounds. Motivated by work of Chia and Thomassen~\cite{CT12} and a question of Zamfirescu~\cite{Za22}, we show in Section~\ref{sec:nonham} that the situation is drastically different when one considers longest cycles in non-hamiltonian 2-connected regular graphs, even when we take into account girth constraints. More specifically, we will show for all integers $r,g \geq 3$ and all integers $n$ for which $nr$ is even and $n$ is large enough that $h_2(n,r,g)$ is bounded from above by a constant depending only on $r$ and $g$. Finally, we formulate questions that are interesting for future research in Section~\ref{sec:conclusions}.

\section{Improved asymptotic upper bounds for $h(n,r)$, where $r \geq 5$} \label{sec:ham}

There exist infinite families of 3-regular graphs containing precisely three pairwise distinct hamiltonian cycles~\cite{Sc89,GMZ20} and also 4-regular graphs with precisely 216~\cite{TZ21} and 144~\cite{Za22} pairwise distinct hamiltonian cycles for all large enough orders such that these graphs exist. This shows that $h(n,3)$ and $h(n,4)$ are asymptotically $O(1)$. However, the situation is quite different for the $r$-regular case when $r \geq 5$, because the best known asymptotic upper bounds are exponential. Haythorpe~\cite{Ha18} constructed graphs which show that $h(n,r) \leq (r-1)^2 ((r-2)!)^{\frac{n}{r+1}}$ for all integers $r \geq 3$ and $n=m(r+1)$ , where $m \geq 2$. Based on empirical evidence, he then also made the following lower bound conjecture.

\begin{conjecture}[Conjecture 3.1 from~\cite{Ha18}]
For $r \geq 5$ and $n \geq r+3$, all hamiltonian $r$-regular graphs on $n$ vertices have at least $(r-1)^2 ((r-2)!)^{\frac{n}{r+1}}$ hamiltonian cycles.
\end{conjecture}

This conjecture was later disproved by Zamfirescu~\cite{Za22} for $r \in \{5,6,7\}$ and by Goedgebeur et al.~\cite{GJLSZ22} for all $r \geq 5$ by constructing graphs which show that $h(n,r) \leq 2((r-1)!)^{r-2}((r-2)!)^{\frac{n-r^2-r+4}{r+1}}$. For fixed $r$ (and $n$ regarded as a free parameter), these upper bounds are all asymptotically $O((\sqrt[r+1]{(r-2)!})^n)$, which is the best known asymptotic upper bound. This raises the natural question of whether this could be optimal. In this section, we will construct graphs that provide exponential improvements to this asymptotic upper bound (i.e. these graphs yield a smaller base for the exponent). Our result also shows that Haythorpe's conjecture does not only involve an incorrect constant factor, but is even asymptotically incorrect.

We first prove a proposition that allows us to count the number of (pairwise distinct) hamiltonian cycles in a graph that is obtained by combining two other graphs.
\begin{proposition}
\label{constructionProposition}
For any integer $r \geq 2$, let $G$ be an $r$-regular graph of order $n_1$ such that it has an edge $e$ that is contained in precisely $c_1$ hamiltonian cycles. Let $H$ be a graph of order $n_2$ that has 2 vertices $u$ and $v$ of degree $r-1$ and $n_2-2$ vertices of degree $r$. Let $c_2$ be the number of hamiltonian $uv$-paths in $H$ and let $e'$ be an edge of $H$ that is contained in precisely $c_3$ hamiltonian $uv$-paths. For every integer $k \geq 1$, there exists an $r$-regular graph $G_k$ on $n := n_1+kn_2$ vertices containing precisely $c_1c_2c_3^{k-1}$ hamiltonian cycles, which is asymptotically $O((\sqrt[n_2]{c_3})^n)$. 
\end{proposition}
\begin{proof}
We set $G_0 := G$ and will consecutively construct graphs $G_1, G_2, \ldots, G_k$. We construct $G_1$ by taking the disjoint union of $G_0$ and $H$, removing the edge $e$ and connecting one of the previous endpoints of $e$ with $u$ and the other endpoint with $v$. Now $G_1$ is an $r$-regular graph on $n_1+n_2$ vertices, which has precisely $c_1c_2$ hamiltonian cycles and the edge $e'$ is contained in precisely $c_1c_3$ hamiltonian cycles. For each integer $i \in \{2, \ldots, k\}$, we apply the same construction to obtain $G_i$ using $G_{i-1}$, the edge $e'$ of $G_{i-1}$ and $H$. The graph $G_k$ has $n := n_1+kn_2$ vertices and it contains precisely $c_1c_2c_3^{k-1}$ hamiltonian cycles. We rewrite this expression to see that it is $O((\sqrt[n_2]{c_3})^n)$: $c_1c_2c_3^{k-1} = c_1c_2c_3^{\frac{n-(n_1+n_2)}{n_2}} = c_1(\sqrt[n_2]{c_3})^{-(n_1+n_2)}(\sqrt[n_2]{c_3})^n$.
\end{proof}

We now describe the construction of a gadget that will later be used as the choice for $H$ in Proposition~\ref{constructionProposition}.

\begin{proposition}
\label{gadgetProposition}
Let $r \geq 5$ be an integer and let $H$ be the graph shown in Fig.~\ref{fig:construction}. $H$ has order $3r+1$, 2 vertices $u$ and $v$ of degree $r-1$, $3r-1$ vertices of degree $r$ and an edge $e'$ contained in precisely $(2r-8)((r-4)!)^2((r-1)!)$ hamiltonian $uv$-paths.
\end{proposition}
\begin{proof}
The set of vertices in $V(H)$ and its neighbors are summarized in Table~\ref{tab:neighborsOverview} to avoid any confusion. Note that the vertices $t_1, t_2, \ldots, t_{r-1}$ are not explicitly drawn in Fig.~\ref{fig:construction}. The dotted circle with the text $K_{r+1}^{-}$ on the bottom of this figure is the graph induced by the $r+1$ vertices $\{u', z_{1}', t_1, t_2, \ldots, t_{r-1}\}$ and it is obtained by removing the edge $u'z_{1}'$ from a clique containing these $r+1$ vertices. The edge $e'$ is defined as the edge $z_1z_2$.
\begin{figure}[h!t]
\centering
\includegraphics[width=.7\textwidth]{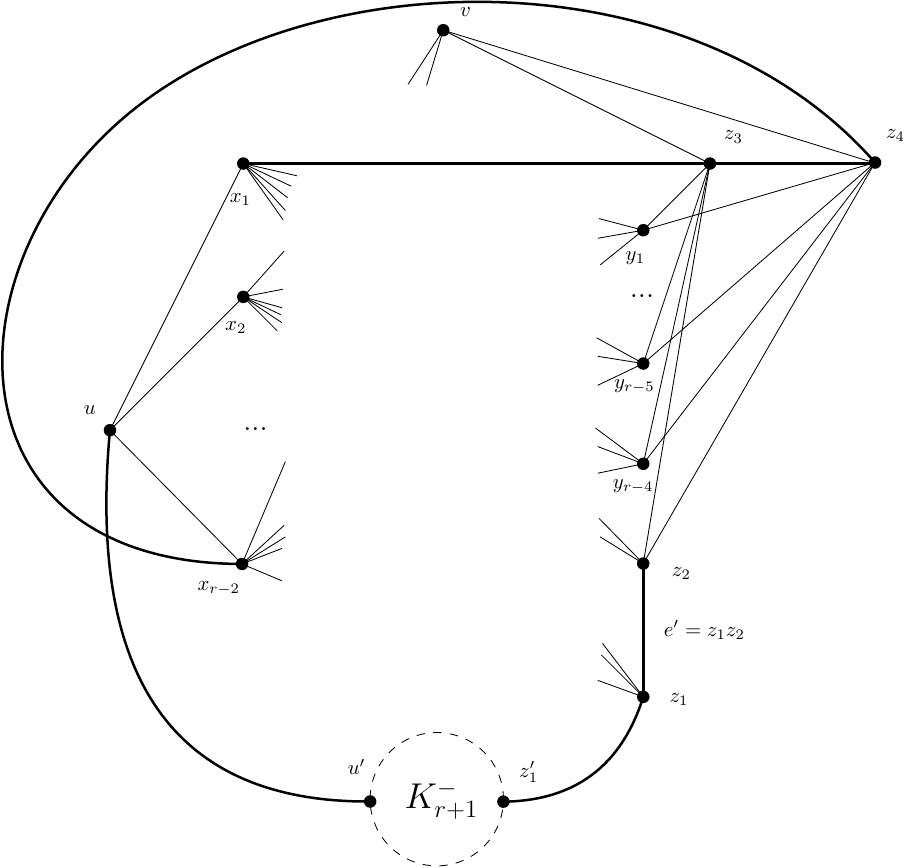}
\caption{The graph $H$ used in Proposition~\ref{gadgetProposition}. If an edge is shown in bold, it is contained in every hamiltonian $uv$-path that contains $e'$. For some edges, only a part of it is drawn to avoid having to draw many crossing edges. Three dots ($\ldots$) correspond to omitted vertices. Edges to omitted vertices are not drawn. The dashed lines on the bottom are not edges, but they represent the graph $K_{r+1}^{-}$, which is the clique $K_{r+1}$ on $r+1$ vertices in which the edge $u'z_{1}'$ is deleted.}
\label{fig:construction}
\end{figure}

\begin{table}[!ht]
  \centering
  \begin{tabular}{|c|c|c|}
    \hline
    \textbf{Vertex} & \textbf{Set of neighbors} & \textbf{Degree} \\
    \hline
    $u$ & $\{u', x_1, x_2, \ldots, x_{r-2}\}$ & $r-1$ \\
    $v$ & $\{z_3,z_4,x_2,x_3,\ldots,x_{r-2}\}$ & $r-1$ \\
    $z_1$ & $\{z_{1}',z_2,x_1, x_2, \ldots, x_{r-2}\}$ & $r$ \\
    $z_2$ & $\{z_1,z_3,z_4,x_1, x_2, \ldots, x_{r-3}\}$ & $r$ \\
    $z_3$ & $\{z_2,z_4,v,x_1, y_1, y_2, \ldots, y_{r-4}\}$ & $r$ \\
    $z_4$ & $\{z_2,z_3,v,x_{r-2}, y_1, y_2, \ldots, y_{r-4}\}$ & $r$ \\
    $u'$ & $\{u, t_1, t_2, \ldots, t_{r-1}\}$ & $r$ \\
    $z_{1}'$ & $\{z_1, t_1, t_2, \ldots, t_{r-1}\}$ & $r$ \\
    $x_{1}$ & $\{u, z_1, z_2, z_3, y_1, y_2, \ldots, y_{r-4}\}$ & $r$ \\
    $x_{i} (2 \leq i \leq r-3)$ & $\{u, z_1, z_2, v, y_1, y_2, \ldots, y_{r-4}\}$ & $r$ \\
    $x_{r-2}$ & $\{u, z_1, z_2, z_4, y_1, y_2, \ldots, y_{r-4}\}$ & $r$ \\
    $y_{i} (1 \leq i \leq r-4)$ & $\{z_3, z_4, x_1, x_2, \ldots, x_{r-2}\}$ & $r$ \\
    $t_{i} (1 \leq i \leq r-1)$ & $\{u', z_{1}', t_1, t_2, \ldots, t_{r-1}\} \setminus \{t_i\}$ & $r$ \\
    \hline
  \end{tabular}
  \caption{An overview of the neighbors of each vertex in $V(H)$ from Fig.~\ref{fig:construction}.}
  \label{tab:neighborsOverview}
\end{table}

%Note that the edge set $\{uu', z_1z_{1}'\}$ is a cutset and therefore the number of hamiltonian $uv$-paths containing $e'$ in $H$ is equal to the number of hamiltonian $u'z_1$-paths in $K_{r+1}^{-}$ multiplied by the number of hamiltonian $z_2v$-paths in the graph induced by the vertices $\{z_2, z_3, z_4, v, x_1, x_2, \ldots, x_{r-2}, y_1, y_2, \ldots, y_{r-4}\}$.

%Note that the edge set $\{uu', z_1z_{1}'\}$ is a cutset and therefore the number of hamiltonian $uv$-paths containing $e'$ in $H$ is equal to the product of the number of hamiltonian $u'z_1$-paths in $K_{r+1}^{-}$ and the number of hamiltonian $z_2v$-paths in the graph induced by the vertices $\{z_2, z_3, z_4, v, x_1, x_2, \ldots, x_{r-2}, y_1, y_2, \ldots, y_{r-4}\}$.  

Note that the edge set $\{uu', z_1z_{1}'\}$ is a cutset and therefore the number of hamiltonian $uv$-paths in $H$ containing $e'$ is equal to the product of the number of hamiltonian $z_2v$-paths in the graph $I$ induced by the vertices $\{z_2, z_3, z_4, v, x_1, x_2, \ldots, x_{r-2}, y_1, y_2, \ldots, y_{r-4}\}$ and the number of hamiltonian $u'z_1'$-paths in $K_{r+1}^{-}$. The second factor is equal to $(r-1)!$, because the vertices $t_1, t_2, \ldots, t_{r-1}$ can be permuted arbitrarily.

For the first factor, we are interested in determining the number of ordered sequences $(a_1=z_2,a_2,\ldots,a_{2r-2}=v)$ of $2r-2$ elements that correspond to hamiltonian $z_2v$-paths in $I$ (i.e. there are $2r-4 = 2(r-2)$ unknowns). Define $X := \{x_1, x_2, \ldots, x_{r-2}\}$. Note that $X$ is an independent set of cardinality $r-2$. Therefore, a vertex in $X$ cannot be preceded nor succeeded by another element from $X$ in the path. Hence, there are in principle only three cases for how the hamiltonian path can look like. For the first (second) case, we have alternatingly $a_2 \in X$ ($a_2 \in V(I) \setminus X$), $a_3 \in V(I) \setminus X$ ($a_3 \in X$), $\ldots$, $a_{2r-3} \in V(I) \setminus X$ ($a_{2r-3} \in X$). However, $z_3$ has $x_1$ as its only neighbor from $X$ and $z_4$ has $x_{r-2}$ as its only neighbor from $X$. This means that there are no hamiltonian $z_2v$-paths in $I$ that correspond to these two cases. For the third case, there exists an even integer $2 \leq i \leq 2r-6$ such that the following three conditions are met:
\begin{enumerate}
\item We have alternatingly $a_2 \in X$, $a_3 \in V(I) \setminus X$, $\ldots$, $a_{i-1} \in V(I) \setminus X$.
\item We have either a) $a_i=x_1, a_{i+1}=z_3, a_{i+2}=z_4$ and $a_{i+3}=x_{r-2}$ or\\ b) $a_i=x_{r-2}, a_{i+1}=z_4, a_{i+2}=z_3$ and $a_{i+3}=x_1$.
\item We have alternatingly $a_{i+4} \in V(I) \setminus X$, $a_{i+5} \in X$, $\ldots$, $a_{2r-3} \in X$.
\end{enumerate}
For case ``a'' of the second condition, there are $r-3$ valid indices $i$. For every fixed $i$, the remaining $r-4$ elements from $X \setminus \{x_1, x_{r-2}\}$ and the remaining $r-4$ elements from $(V(I) \setminus X) \setminus \{z_3,z_4\}$ can each be permuted arbitrarily. Hence, there are $(r-3)((r-4)!)^2$ hamiltonian $z_2v$-paths in $I$ satisfying case ``a'' of the second condition. For case ``b'' of the second condition, $i$ cannot be equal to 2 (because $x_{r-2}$ is not a neighbor of $z_2$) and $i$ cannot be equal to $2r-6$ (because $x_1$ is not a neighbor of $v$). Hence, there are $r-5$ valid indices $i$ and, again, the remaining $r-4$ elements from $X \setminus \{x_1, x_{r-2}\}$ and the remaining $r-4$ elements from $(V(I) \setminus X) \setminus \{z_3,z_4\}$ can each be permuted arbitrarily, which leads to $(r-5)((r-4)!)^2$ hamiltonian $z_2v$-paths in $I$ satisfying case ``b'' of the second condition. We conclude that there are $(2r-8)((r-4)!)^2$ hamiltonian $z_2v$-paths in $I$ and thus there are precisely $(2r-8)((r-4)!)^2((r-1)!)$ hamiltonian $uv$-paths in $H$ containing $e'$.
\end{proof}

We now obtain the following asymptotic upper bound:

\begin{theorem}
For every integer $r \geq 5$, $h(n,r)$ is asymptotically $$O((\sqrt[3r+1]{(2r-8)((r-4)!)^2((r-1)!)})^n).$$
\end{theorem}
\begin{proof}
We use the graph described in Proposition~\ref{gadgetProposition} as our choice for the graph $H$ that is required in Proposition~\ref{constructionProposition}. We have $n_2=3r+1$ and $c_3=\sqrt[3r+1]{(2r-8)((r-4)!)^2((r-1)!)}$. Note that the asymptotic upper bound that Proposition~\ref{constructionProposition} gives us, does not depend on the choice of $G$, so any $r$-regular graph (e.g. the clique $K_{r+1}$) can be chosen for $G$.
\end{proof}

Table~\ref{tab:comparisonUpperBounds} compares these new asymptotic upper bounds with the previous best ones and shows that they are indeed better for $r \in \{5,6,7,8\}$. 

\begin{table}[!ht]
  \centering
  \renewcommand{\arraystretch}{1.05}
  \begin{tabular}{|c|c|c|}
    \hline
    $\mathbf{r}$ & \textbf{Previous best asymptotic upper bound} & \textbf{New asymptotic upper bound} \\
    \hline
    $5$ & $O((\sqrt[6]{6})^n) \approx O(1.348^n)$ & $O((\sqrt[16]{48})^n) \approx O(1.274^n)$ \\
    $6$ & $O((\sqrt[7]{24})^n) \approx O(1.575^n)$ & $O((\sqrt[19]{1920})^n) \approx O(1.489^n)$ \\
    $7$ & $O((\sqrt[8]{120})^n) \approx O(1.819^n)$ & $O((\sqrt[22]{155,520})^n) \approx O(1.722^n)$ \\
    $8$ & $O((\sqrt[9]{720})^n) \approx O(2.077^n)$ & $O((\sqrt[25]{23,224,320})^n) \approx O(1.971^n)$ \\
    \hline
  \end{tabular}
  \caption{A comparison between the previous best asymptotic upper bound and the new one for $h(n,5), h(n,6), h(n,7)$ and $h(n,8)$.}
  \label{tab:comparisonUpperBounds}
\end{table}

The following proposition shows that the upper bounds are in fact better for all $r \geq 5$ by proving that the new base of the exponent is smaller than the previous best base. 

\begin{proposition}
For every integer $r \geq 5$, we have $\sqrt[3r+1]{(2r-8)((r-4)!)^2((r-1)!)} < \sqrt[r+1]{(r-2)!}$.
\end{proposition}
\begin{proof}
We first bring the inequality to a more convenient form:
\begin{align*}
\sqrt[3r+1]{(2r-8)((r-4)!)^2((r-1)!)} &< \sqrt[r+1]{(r-2)!} &\Leftrightarrow \\
(2r-8)^{r+1}((r-4)!)^{2r+2}(r-1)^{r+1}((r-2)!)^{r+1} &< ((r-2)!)^{3r+1} &\Leftrightarrow \\
(2r-8)^{r+1}((r-4)!)^{2r+2}(r-1)^{r+1} &< (r-2)^{2r}(r-3)^{2r}((r-4)!)^{2r} &\Leftrightarrow \\
(2r-8)^{r+1}((r-3)!)^{2}(r-1)^{r+1} &< (r-2)^{2r}(r-3)^{2r+2} &\Leftrightarrow \\
((2r-8)(r-1))^{r+1}((r-3)!)^{2} &< ((r-3)^{2})^{r+1}(r-2)^{2r}
\end{align*}
We now bound the left-hand side by using Stirling's formula:
\begin{align*}
((2r-8)(r-1))^{r+1}((r-3)!)^{2} &< e^{r+1}\Bigg(\frac{(2r-8)(r-1)}{e}\Bigg)^{r+1}\Bigg(\sqrt{2\pi(r-3)e}\Bigg(\frac{r-3}{e}\Bigg)^{r-3}\Bigg)^2\\
&= \Bigg(\frac{(2r-8)(r-1)}{e}\Bigg)^{r+1} (2\pi(r-3)e^{8-r})(r-3)^{2r-6}
\end{align*}
We proved the inequality:
$$ \Bigg(\frac{(2r-8)(r-1)}{e}\Bigg)^{r+1} (2\pi(r-3)e^{8-r})(r-3)^{2r-6} < ((r-3)^{2})^{r+1}(r-2)^{2r}$$
for $5 \leq r \leq 12$ with a computer and for $r \geq 13$ note that the quadratic polynomial $\frac{(2r-8)(r-1)}{e}$ is smaller than $(r-3)^{2}$, that $2\pi(r-3)e^{8-r}<1$ and that  $(r-3)^{2r-6}<(r-2)^{2r}$.
\end{proof}

\section{Constant asymptotic upper bounds for $h_2(n,r,g)$ for all integers $r,g \geq 3$} \label{sec:nonham}

We saw before that the minimum number of longest cycles in hamiltonian 3-regular and 4-regular graphs are bounded by a constant, whereas the best known asymptotic upper bounds for hamiltonian $r$-regular graphs ($r \geq 5$) are exponential. It turns out that the situation is similar for $r \in \{3,4\}$, but quite different for $r \geq 5$ when one focuses on the minimum number of longest cycles in non-hamiltonian $r$-regular graphs instead. Chia and Thomassen~\cite{CT12} showed that there are infinitely many $n$ such that there is a 3-regular graph on $n$ vertices containing a unique longest cycle. Zamfirescu~\cite{Za22} complemented this result by showing for each integer $r \geq 3$ the existence of infinitely many non-hamiltonian $r$-regular graphs containing a unique longest cycle. For the 3-regular case, the graphs described by Chia and Thomassen and Zamfirescu respectively have vertex connectivity 2 and 3, whereas for the $r$-regular case ($r \geq 4$) they have connectivity 1. Hence, the 2-connected case remains largely open. Another natural condition in this context is to consider the girths of the graphs. For example, Zamfirescu~\cite{Za22} asked whether there exist 3-regular 2-connected triangle-free non-hamiltonian graphs with a unique longest cycle (in other words: is $h_2(n,3,g)$ equal to 1 for some $n, g \geq 4$), whereas Cantoni conjectured that there are no (2-connected) triangle-free planar 3-regular graphs with exactly three hamiltonian cycles (see~\cite{Tu76}). In this section we will prove that for all integers $r,g \geq 3$, there exist constants $n_{r,g}$ and $c_{r,g}$ depending only on $r$ and $g$ such that for all integers $n \geq n_{r,g}$ for which $nr$ is even, the minimum number of longest cycles in non-hamiltonian 2-connected $r$-regular graphs with girth $g$ is bounded by $c_{r,g}$ (i.e. $h_2(n,r,g)$ is $O(1)$). Motivated by Zamfirescu's question, we also give an example of a 3-regular 2-connected triangle-free non-hamiltonian graph containing precisely four pairwise distinct longest cycles.

We first require a classical theorem by Sachs.

\begin{theorem}[Theorem 1 from~\cite{Sa63}]
\label{SachsTheorem}
For all integers $r, g \geq 3$, there exists a hamiltonian $r$-regular graph with girth $g$.
\end{theorem}

We now obtain the following corollary:

\begin{corollary}
\label{edgeRemovalCorollary}
For all integers $r, g \geq 3$, there exists a graph $G'$ which has two distinct vertices $u$ and $v$ of degree $r-1$ and all other vertices of degree $r$ such that 1) the distance between $u$ and $v$ is at least $g-1$, 2) $G'$ has a hamiltonian $uv$-path and 3) $G'$ has girth $g$.
\end{corollary}
\begin{proof}
Let $G$ be a hamiltonian $r$-regular graph with girth $g$ (note that $G$ always exists because of Theorem~\ref{SachsTheorem}). Let $c_1$ be a cycle of $G$ containing precisely $g$ edges and let $c_2$ be a hamiltonian cycle of $G$. Let $uv$ be an edge from the edge set $E(c_2) \setminus E(c_1)$ (note that this set is not empty, because $g < |V(G)|$). Construct the graph $G'$ by removing the edge $uv$ from $G$. In the graph $G'$, the distance between $u$ and $v$ is at least $g-1$ (because $G$ has girth $g$), $G'$ has a hamiltonian $uv$-path (with edge set $E(c_2) \setminus \{uv\}$) and $G'$ has girth $g$ (the cycle $c_1$ has length $g$).
\end{proof}

Based on this, we now prove the existence of two $r$-regular graphs with girth $g$ such that the difference between their orders is the smallest possible positive integer.
 
\begin{theorem}
\label{consecutiveTheorem}
For all integers $r, g \geq 3$, there exist integers $n_1$ and $n_2$ such that $G_i$ is a hamiltonian $r$-regular graph with girth $g$ on $n_i$ vertices ($i \in \{1,2\}$) and $n_2=n_1+1$ if $r$ is even and $n_2=n_1+2$ if $r$ is odd.
\end{theorem}
\begin{proof}
Let $G'$ be a graph satisfying the conditions mentioned in Corollary~\ref{edgeRemovalCorollary}.

We first deal with the case when $r$ is even. Let $G_0', G_1', \ldots, G_{r-1}'$ be $r$ isomorphic copies of $G'$, call the vertices of degree $r-1$ in these graphs $u_0', v_0', u_1', v_1', \ldots, u_{r-1}', v_{r-1}'$ and call the hamiltonian $u_i'v_i'$-paths in these graphs $h_i'$ ($i \in \{0, 1, \ldots, r-1\}$). Now we construct the graph $G_1$ by taking the disjoint union of $G_0', G_1', \ldots, G_{r-1}'$  and adding $r$ edges $e_0, e_1, \ldots, e_{r-1}$ between respectively vertex $v_i'$ and $u_{i+1}'$, where indices are taken modulo $r$ ($i \in \{0, 1, \ldots, r-1\}$). $G_1$ is a hamiltonian $r$-regular graph with girth $g$ on $n_1 := r|V(G')|$ vertices. The edge set of the hamiltonian cycle in $G_1$ is the union of $\{e_0, e_1, \ldots, e_{r-1}\}$ and the edge sets $E(h_i')$ ($i \in \{0, 1, \ldots, r-1\}$). Let $a_0'$ be an edge in $E(h_0')$ between vertices $s_0'$ and $t_0'$ and let $b_i'$ be an edge in $E(G_i') \setminus E(h_i')$ between vertices $s_i'$ and $t_i'$ ($i \in \{2,4,\ldots,r-2\}$). The graph $G_2$ is obtained by adding a new vertex $w$ to $G_1$ and removing the edges $a_0', b_2', b_4', \ldots, b_{r-2}'$ and adding $r$ edges between respectively $w$ and $s_0', t_0', s_2', t_2', \ldots, s_{r-2}', t_{r-2}'$. A schematic representation of $G_2$ for $r=4$ is given in Fig.~\ref{fig:evenCaseFigCorollary}. $G_2$ is clearly an $r$-regular graph on $n_2 := r|V(G')|+1$ vertices. $G_2$ has girth $g$ because $G_1'$ contains a cycle of length $g$ and the distance between any two vertices from the set $\{s_0', t_0', s_2', t_2', \ldots, s_{r-2}', t_{r-2}'\}$ in the graph $G_2-w$ is at least $g-1$. Finally, $G_2$ also has a hamiltonian cycle with edge set:
$$\Big(E(h_0') \cup E(h_1') \cup \ldots \cup E(h_{r-1}') \cup \{e_0, e_1, \ldots, e_{r-1}\} \cup \{ws_0',wt_0'\}\Big) \setminus \{a_0'\}.$$

\begin{figure}[h]
%\begin{center}
\centering
\begin{subfigure}{0.3\linewidth}
\centering\includegraphics[width=1.1\linewidth]{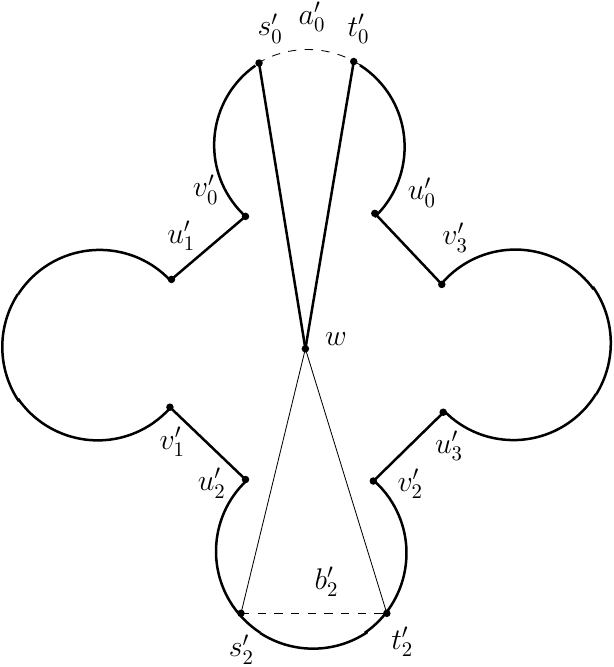}
\caption{}
\label{fig:evenCaseFigCorollary}
\end{subfigure}%
\hspace{0.7cm}
\begin{subfigure}{0.55\linewidth}
\centering\includegraphics[width=1.1\linewidth]{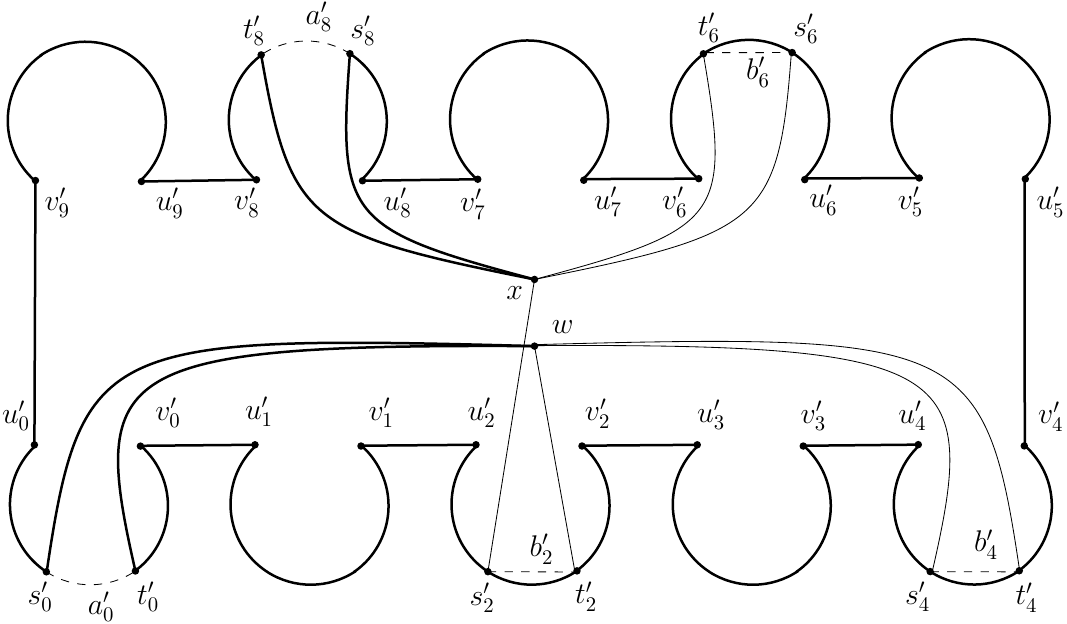}
\caption{}
\label{fig:oddCaseFigCorollary}
\end{subfigure}
\caption{A schematic representation of the graph $G_2$ for (a) $r=4$ and (b) $r=5$. A dashed line is an edge that is present in $G_1$, but not in $G_2$. The bold lines indicate the hamiltonian cycle in $G_2$. Each circular sector represents an isomorphic copy of $G'$.}
\label{fig:twoCases}
%\centering \\Fig.~3: EXPLANATION HERE.
%\end{center}
\end{figure}

We now deal with the case when $r$ is odd. Let $G_0', G_1', \ldots, G_{2r-1}'$ be $2r$ isomorphic copies of $G'$, call the vertices of degree $r-1$ in these graphs $u_0', v_0', u_1', v_1', \ldots, u_{2r-1}', v_{2r-1}'$ and call the hamiltonian $u_i'v_i'$-paths in these graphs $h_i'$ ($i \in \{0, 1, \ldots, 2r-1\}$). Now we construct the graph $G_1$ by taking the disjoint union of $G_0', G_1', \ldots, G_{2r-1}'$  and adding $2r$ edges $e_0, e_1, \ldots, e_{2r-1}$ between respectively vertex $v_i'$ and $u_{i+1}'$, where indices are taken modulo $2r$ ($i \in \{0, 1, \ldots, 2r-1\}$). $G_1$ is a hamiltonian $r$-regular graph with girth $g$ on $n_1 := 2r|V(G')|$ vertices. The edge set of the hamiltonian cycle in $G_1$ is the union of $\{e_0, e_1, \ldots, e_{2r-1}\}$ and the edge sets $E(h_i')$ ($i \in \{0, 1, \ldots, 2r-1\}$). Let $a_i'$ be an edge in $E(h_i')$ between vertices $s_i'$ and $t_i'$ ($i \in \{0,2r-2\}$) and let $b_i'$ be an edge in $E(G_i') \setminus E(h_i')$ between vertices $s_i'$ and $t_i'$ ($i \in \{2, 4, \ldots, 2r-4\}$). The graph $G_2$ is obtained by adding two new vertices $w$ and $x$ to $G_1$ and removing the edges $a_0', a_{2r-2}', b_2', b_4', \ldots, b_{2r-4}'$ and adding $r$ edges between respectively $w$ and $s_0', t_0', t_2', s_4', t_4', s_6', t_6', \ldots, s_{r-1}', t_{r-1}'$ and adding $r$ edges between respectively $x$ and $s_{2r-2}', t_{2r-2}', s_2', s_{r+1}', t_{r+1}', s_{r+3}', t_{r+3}', \ldots, s_{2r-4}', t_{2r-4}'$. A schematic representation of $G_2$ for $r=5$ is given in Fig.~\ref{fig:oddCaseFigCorollary}. $G_2$ is clearly an $r$-regular graph on $n_2 := 2r|V(G')|+2$ vertices. $G_2$ has girth $g$ because $G_1'$ contains a cycle of length $g$ and the distance between any two vertices from the set $\{s_0', t_0', s_2', t_2', \ldots, s_{2r-2}', t_{2r-2}'\}$ in the graph $G_2-w-x$ is at least $g-1$. Finally, $G_2$ also has a hamiltonian cycle with edge set:
$$\Big(E(h_0') \cup E(h_1') \cup \ldots \cup E(h_{2r-1}') \cup \{e_0, e_1, \ldots, e_{2r-1}\} \cup \{ws_0',wt_0',xs_{2r-2}',xt_{2r-2}'\}\Big) \setminus \{a_0', a_{2r-2}'\}.$$
\end{proof}

As an easy corollary, we obtain an infinite family of $r$-regular graphs with girth $g$.

\begin{corollary}
\label{infinitelyManyCorollary}
For all integers $r, g \geq 3$ there are infinitely many hamiltonian $r$-regular graphs with girth $g$.
\end{corollary}
\begin{proof}
In the proof of Theorem~\ref{consecutiveTheorem}, we started from an arbitrary hamiltonian $r$-regular graph $G$ with girth $g$ to construct larger hamiltonian $r$-regular graphs $G_1$ and $G_2$ with girth $g$. We can repeat this process, this time constructing larger graphs $G_3$ and $G_4$ starting from $G=G_2$ and so on.
\end{proof}

We now use the previous theorems and corollaries to arrive at the main theorem of this section.

\begin{theorem}
For all integers $r,g \geq 3$, there exist constants $n_{r,g}$ and $c_{r,g}$ depending only on $r$ and $g$ such that for all integers $n \geq n_{r,g}$ for which $nr$ is even, we have $h_2(n,r,g) \leq c_{r,g}$.
\end{theorem}
\begin{proof}
Let $G_1$ and $G_2$ be hamiltonian $r$-regular graphs with girth $g$ on respectively $n_1$ and $n_2$ vertices such that $n_2=n_1+1$ if $r$ is even and $n_2=n_1+2$ is $r$ is odd (such graphs exist because of Theorem~\ref{consecutiveTheorem}). Let $G_3$ be a hamiltonian $r$-regular graph with girth $g$ on $n_3>n_2$ vertices (such a graph exists because of Corollary~\ref{infinitelyManyCorollary}). For $i \in \{1, 2, 3\}$, construct the graph $G_i'$ on $n_i$ vertices which has two distinct vertices $u_i$ and $v_i$ of degree $r-1$ and all other vertices of degree $r$ such that 1) the distance between $u$ and $v$ is at least $g-1$, 2) $G_i'$ has a hamiltonian $u_{i}v_{i}$-path and 3) $G_i'$ has girth $g$ ($G_i'$ can be obtained by applying the construction mentioned in the proof of Corollary~\ref{edgeRemovalCorollary} to $G_i$). Construct the graph $H_1$ by taking the disjoint union of $r-2$ copies of $G_1'$ and adding two vertices $u_1$ and $v_1$ and edges between $u_1$ and one vertex of degree $r-1$ in each of the $r-2$ isomorphic copies of $G_1'$ and edges between $v_1$ and the other vertex of degree $r-1$ in each of those copies (i.e. $u_1$ and $v_1$ have degree $r-1$ and all other vertices have degree $r$). In a similar fashion, construct the graph $H_2$, but this time replace one of the $r-2$ copies of $G_1'$ by a copy of $G_2'$. Note that $|V(H_1)|=(r-2)|V(G_1')|+2$ and that $|V(H_2)|=|V(H_1)|+1$ when $r$ is even and $|V(H_2)|=|V(H_1)|+2$ when $r$ is odd. For non-negative integers $\ell$ and $m$ satisfying $\ell+m \geq 1$, we construct the graph $G_{\ell,m}$ as follows. We take the disjoint union of $\ell$ copies $H_{1,1}, H_{1,2}, \ldots, H_{1,\ell}$ of $H_1$, $m$ copies $H_{2,1}, H_{2,2}, \ldots, H_{2,m}$ of $H_2$ and two copies $G_{3,1}'$ and $G_{3,2}'$ of $G_3'$. Let $w_1, x_1, w_{2+\ell+m}$ and $x_{2+\ell+m}$ be the two vertices of degree $r-1$ in respectively $G_{3,1}'$ and $G_{3,2}'$ and let $w_{1+j}$ and $x_{1+j}$, and $w_{1+\ell+k}$ and  $x_{1+\ell+k}$ be the two vertices of degree $r-1$ in respectively $H_{1,j}$ ($j \in \{1, 2, \ldots, \ell\}$) and $H_{2,k}$ ($k \in \{1, 2, \ldots, m\}$). Now $G_{\ell,m}$ is obtained by adding an edge $w_{i}w_{i+1}$ and an edge $x_{i}x_{i+1}$ for each $i \in \{1,2,\ldots,1+\ell+m\}$. Fig.~\ref{fig:constantUpperBound} shows a schematic representation of the graph $G_{1,1}$ for $r=4$.

\begin{figure}[h!t]
\centering
\includegraphics[width=.7\textwidth]{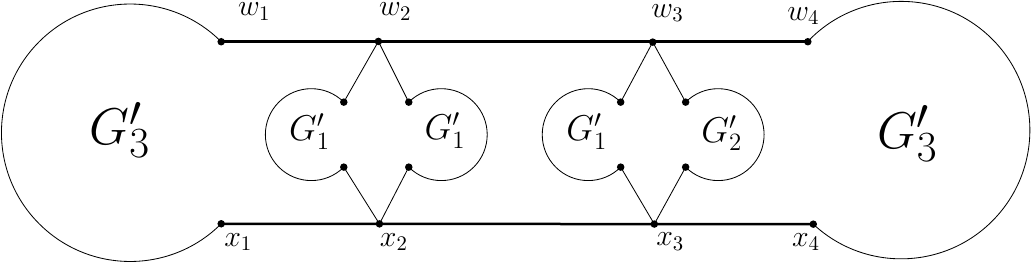}
\caption{A schematic representation of the graph $G_{1,1}$ for $r=4$. The circular sectors represent an isomorphic copy of $G_{i'}$ (for some $i \in \{1,2,3\}$). The bold edges are contained in every longest cycle of $G_{1,1}$.}
\label{fig:constantUpperBound}
\end{figure}

The graph $G_{\ell,m}$ is clearly $r$-regular. $G_{\ell,m}$ also has girth $g$, because $G_1'$, $G_2'$ and $G_3'$ all have girth $g$ and the distance between the two vertices of degree $r-1$ in each copy of $G_1'$, $G_2'$ or $G_3'$ is at least $g-1$. $G_{\ell,m}$ has vertex connectivity 2 (e.g. $\{w_1, x_1\}$ is a vertex cut), because none of the vertices from $\{w_1,x_1,w_2,x_2,\ldots,w_{2+\ell+m},x_{2+\ell+m}\}$ are cut vertices. Moreover, for each copy of $G_1'$, $G_2'$ or $G_3'$, there does not exist a vertex $y$ such that its removal results in a connected component that contains both vertices of degree $r-1$, because $G_1'$, $G_2'$ and $G_3'$ contain a hamiltonian path between these two vertices.

There does not exist a cycle in $G_{\ell,m}$ containing vertices from three or more copies of $G_i'$ ($i \in \{1,2,3\}$). Therefore, $G_{\ell,m}$ is not hamiltonian and every longest cycle of $G_{\ell,m}$ has length $2(\ell+m+|V(G_3')|)$. Every such cycle contains the edges $w_{i}w_{i+1}$ and $x_{i}x_{i+1}$ ($i \in \{1,2,\ldots,1+\ell+m\}$) and the edges of a hamiltonian path (between the two vertices of degree $r-1$ in the graph $G_3'$) in each of the two copies of $G_3'$. If $h$ is the number of pairwise distinct hamiltonian paths between the two vertices of degree $r-1$ in $G_3'$, then the number of longest cycles in $G_{\ell,m}$ is equal to $c_{r,g} := h^2$ (note that this does not depend on the choice of $\ell$ and $m$).

Let $a$ be equal to $1$ if $r$ is even and equal to $2$ if $r$ is odd. The graph $G_{\ell,m}$ has $2|V(G_3')|+(\ell+m)((r-2)|V(G_1')|+2)+am$ vertices. Define $n_{r,g} := 2|V(G_3')|+((r-2)|V(G_1')|+2)^2$. For all integers $n \geq n_{r,g}$ such that $nr$ is even, let $q_1$  and $0 \leq r_1 < (r-2)|V(G_1')|+2$ be integers such that $n=q_1((r-2)|V(G_1')|+2)+r_1$. Note that $a$ is a divisor of $r_1$. If we choose $\ell=q_1-\frac{r_1}{a}$ and $m=\frac{r_1}{a}$, then the graph $G_{\ell,m}$ is a graph on $n$ vertices containing $c_{r,g}$ pairwise distinct longest cycles and thus $h_2(n,r,g) \leq c_{r,g}$.
\end{proof}

We give an example of a 3-regular 2-connected triangle-free non-hamiltonian graph containing precisely four pairwise distinct longest cycles in Fig.~\ref{fig:zamfirescuQuestion}. This can be seen by realizing that the set of vertices $\{1, 10\}$ is a cutset and the graph induced by the vertices $2, 3, \ldots ,9$ has precisely two distinct hamiltonian paths between vertex 2 and 9 (and similarly for the graph induced by the vertices $11, 12, \ldots, 18$). The graph has been made available on \textit{House of Graphs} \cite{CDG23} at \url{https://houseofgraphs.org/graphs/50421}.

\begin{figure}[h!t]
\centering
\includegraphics[width=.7\textwidth]{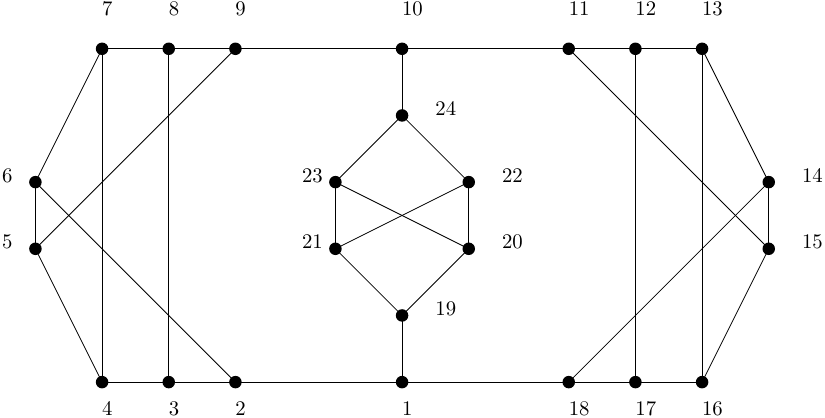}
\caption{A 3-regular 2-connected triangle-free non-hamiltonian graph on twenty-four vertices containing precisely four pairwise distinct longest cycles (of length 18): a) 1, 2, 3, 4, 5, 6, 7, 8, 9, 10, 11, 12, 13, 14, 15, 16, 17, 18, 1; b) 1, 2, 3, 4, 5, 6, 7, 8, 9, 10, 11, 15, 16, 17, 12, 13, 14, 18, 1; c) 1, 2, 6, 7, 8, 3, 4, 5, 9, 10, 11, 12, 13, 14, 15, 16, 17, 18, 1 and d) 1, 2, 6, 7, 8, 3, 4, 5, 9, 10, 11, 15, 16, 17, 12, 13, 14, 18, 1.}
\label{fig:zamfirescuQuestion}
\end{figure}

We computationally verified, using two independent algorithms, that there are no 3-regular 2-connected triangle-free non-hamiltonian graphs on at most 28 vertices containing less than four pairwise distinct longest cycles. We refer the interested reader to the Appendix for more details about the computations.
 
\section{Concluding remarks} \label{sec:conclusions}

Finally, we would like to consider relaxations of a question asked by Zamfirescu~\cite{Za22}: are there 3-regular 2-connected triangle-free non-hamiltonian graphs containing a unique longest cycle? Our computations imply that such a graph has to contain at least 30 vertices if it exists. We ask more generally: are there integers $r \geq 3$ and $n,g \geq 4$ for which $h_2(n,r,g)=1$? We are also interested in graphs that are close to being 3-regular. More specifically, an interesting relaxation is obtained by relaxing the 3-regular constraint: amongst all 2-connected triangle-free non-hamiltonian graphs with minimum degree 3 containing a unique longest cycle, what is the minimum number of vertices with a degree different from 3? We note that one can construct such a graph on sixty vertices containing fifty-six 3-valent vertices and four 4-valent vertices. More specifically, Royle~\cite{Ro17} constructed a triangle-free graph with sixteen 3-valent and two 4-valent vertices containing a unique hamiltonian cycle (further examples are given in~\cite{GMZ20} and~\cite{Se15}), whereas Fig.~\ref{fig:zamfirescuQuestion} shows a 3-regular 2-connected triangle-free non-hamiltonian graph on twenty-four vertices with a unique longest cycle containing simultaneously the edges between vertices 2 and 3 and vertices 17 and 18. If we take two copies of Royle's graph, remove arbitrary edges $uv$ and $u'v'$ (one in each copy) on the unique hamiltonian cycle of this graph, remove the edges between vertices 2 and 3 and vertices 17 and 18 from the graph shown in Fig.~\ref{fig:zamfirescuQuestion} and add four edges between vertices 2 and $u$, 3 and $v$, 17 and $u'$ and 18 and $v'$, we obtain the desired graph on $18+18+24=60$ vertices. The graph has been made available at \url{https://houseofgraphs.org/graphs/50422}. We also note that if one replaces ``minimum degree three'' by ``maximum degree three'', one can subdivide the edge between vertices 2 and 3 and the edge between vertices 17 and 18 in the graph shown in Fig.~\ref{fig:zamfirescuQuestion} to obtain a 2-connected triangle-free non-hamiltonian graph containing a unique longest cycle with two vertices that are not 3-valent.

\bigskip

\noindent\textbf{Acknowledgements.} The author would like to thank Carol T. Zamfirescu and Jan Goedgebeur for a useful discussion about this work. The author is a Postdoctoral Fellow of the Research Foundation Flanders (FWO) with contract number 1222524N. The computational resources and services used in this work were provided by the VSC (Flemish Supercomputer Center), funded by the Research Foundation - Flanders (FWO) and the Flemish Government -- department EWI.

\section*{Appendix}\label{sec: appendix}

We used the graph generator \textit{snarkhunter}~\cite{BGM11} for generating all 3-regular 2-connected triangle-free graphs up to 28 vertices. We implemented two independent algorithms for enumerating all pairwise distinct longest cycles of a graph and used 300 cores from a computer cluster to enumerate such cycles for all generated graphs. The first algorithm consecutively considers the graphs induced by the first $3, 4, \ldots, n$ vertices, where $n$ denotes the order of the graph, and enumerates all cycles containing the highest numbered vertex. It uses a simple recursion scheme that starts from a path $P$ containing three vertices and recursively extends one end of $P$ in all possible ways. The recursion stops as soon as there is no path using only unvisited vertices between the two degree 1 vertices in the graph induced by the vertices of $P$. The second algorithm is a modification of previous code used in~\cite{JLD20} for enumerating all hamiltonian cycles of a graph using Held-Karp's algorithm~\cite{HK62}. Both algorithms were in agreement with each other. The code is made publicly available on GitHub~\cite{Jooken23}.
	
\end{document}